\documentclass[reqno,12pt]{amsart}
\usepackage{amsmath,amsfonts,amsthm,amssymb,amscd,float,color,multirow}
\usepackage[mathscr]{eucal}
\usepackage{slashbox}
\usepackage[dvips]{graphicx}
\usepackage[all]{xy}

\theoremstyle{plain}
\newtheorem{thm}{Theorem}
\newtheorem{lem}[thm]{Lemma}

\newtheorem{prop}[thm]{Proposition}

  \newtheorem{thma}{Theorem A}
  
    \newtheorem{thmb}{Theorem B}
  
  \newtheorem{propc}{Proposition C}
  
   \newtheorem{cord}{Corollary D}
  
\newtheorem{thme}{Theorem E}
  
  \newtheorem{thmf}{Theorem F}

\theoremstyle{definition}
  \newtheorem{defn}[thm]{Definition}
  \newtheorem{ex}[thm]{Example}
    \newtheorem{rem}[thm]{Remark}

\newcommand{\N}{{\mathbb N}}

\newcommand{\R}{{\mathbb R}}

\newcommand{\sbar}{\,\;\vert\,\;}
\newcommand{\longsbar}{\,\;\mathstrut\vrule\,\;}

\begin{document}
 
\title[]{ The Euler characteristic of the regular spherical polygon spaces}
\author[]{Yasuhiko KAMIYAMA}
\address[]{Department of Mathematics, University of the Ryukyus, 
Nishihara-Cho, Okinawa 903-0213, Japan} 
\email{kamiyama@sci.u-ryukyu.ac.jp} 

\begin{abstract} 
Let $a$ be a real number satisfying $0<a<\pi$. We denote by $M_n(a)$ the configuration space of regular spherical $n$-gons
with side-lengths $a$. The purpose of this paper is to determine $\chi (M_n(a))$ for all $a$ and odd $n$.
To do so, we construct a manifold $X_n$ and a function $\mu: X_n \to \R$ such that $\mu^{-1}(a)=M_n(a)$. 
In fact, the function $\mu$ is different from the well-known ``wall-crossing" function. 
We determine the index of each critical point of $\mu$. Since a level set is obtained by successive Morse surgeries, we can determine $\chi (M_n(a))$. 
\end{abstract}
\maketitle

\section{Introduction} 
Recently, the topology of polygon spaces in the Euclidean space of dimension two or three has been considered by many authors.
The study of planar polygon spaces started in \cite{H,KM1,W}. For example, the homology groups were determined in \cite{FS}. 
On the other hand, the study of spatial polygon spaces started in \cite{KM2}. 

Morse theory plays a key role in the study of polygon spaces. \cite{F} is an excellent exposition about polygon spaces with emphasis on Morse theory. 
In \cite{MT}, Milgram and Trinkle obtained results by making excellent use of Morse surgery. 

Later, Kapovich and Millson (\cite{KM3}) studied spherical polygon spaces. They determined the index of the ``wall-crossing" Morse function. Using this, they determined the topological type of regular spherical pentagon spaces.  The purpose to this paper is to determine the Euler characteristic of 
regular spherical polygon spaces. For that purpose, we consider a Morse function which is different from the wall-crossing function. 

We give the definition of regular polygon spaces. 
Let $a$ be a real number satisfying $0<a<\pi$ and $n$ be an integer $\geq 3$ . We set
\begin{align}
A_n (a):= \{P=(u_1,\dots,u_n)\in (S^2)^n \sbar d(u_i,&u_{i+1})=a \;\; \text{for $1 \leq i \leq n-1$}\notag\\
& \text{and}\;\; d(u_n,u_1)=a\}. 
\end{align}
Here $d$ is the spherical distance. Let $SO(3)$ act on $A_n(a)$ diagonally and we set
\begin{equation}
M_n(a):= A_n(a)/SO(3). 
\end{equation}

When $n$ is even, $M_n(a)$ has critical points for all $a$. More precisely, singular points are given by degenerate polygons, 
where a degenerate polygon is defined in 
Definition \ref{degenerate}.
Hence, in what follows, we will always assume $n$ to be odd and set $n=2m+1$. 

Let us realize $M_n (a)$ as a level set. 
We first set
\begin{align*}
B_n:=\{(P,a)=((u_1,\dots,u_n),a)\in &(S^2)^n\times (0,\pi) \sbar d(u_i,u_{i+1})=a \notag \\
&\text{for $1 \leq i \leq n-1$}\;\; \text{and}\;\; d(u_n,u_1)=a\}.
\end{align*}
Let $SO(3)$ act on $B_n$ by
$$(P,a)g:= (Pg,a),$$
where $g \in SO(3)$. Then we set
\begin{equation*}
X_n:= B_n/SO(3).
\end{equation*}
We define the function $\mu: X_n \to \R$ by 
\begin{equation}\label{4}
\mu(P,a)=a. 
\end{equation}
Note that for all $a \in (0,\pi)$, we have 
\begin{equation}\label{5}
\mu^{-1}(a)=M_n (a).
\end{equation}

The goal of this paper is to prove that $\mu$ is a Morse function. Since a level set is obtained by successive Morse surgeries, we can determine
$\chi (M_n(a))$ for all $a$. 

We study how our function $\mu$ is different from the well-known wall-crossing Morse function.  Although \cite{KM3} considers 
the spherical polygon spaces of arbitrary side-lengths, we state in the form which fits our situation better. 
We first set
\begin{equation*}
C_n (a) :=  \{T=(u_1,\dots,u_n)\in (S^2)^n \sbar d(u_i,u_{i+1})=a \;\; \text{for $1 \leq i \leq n-1$}\}.
\end{equation*}
Then we set
\begin{equation*}
N_n (a):=C_n(a) /SO(3).
\end{equation*}
We define the function $\rho_a: N_n(a) \to \R$ by 
\begin{equation}\label{g}
\rho_a(u_1,\dots,u_n)=d(u_n,u_1). 
\end{equation}
It is proved in \cite{KM3} that $\rho_a$ is a Morse function. (See Theorem \ref{KM3} for more details.) 
If we want to obtain information on all regular spherical polygon spaces from the main result in \cite{KM3}, 
we need to consider the function $\rho_a$ for various $a$. 
For example, \cite[\S6]{KM3} determined the topological type of all regular spherical pentagon spaces by using this method.

On the other hand, for our case, 
it will suffice to consider just one manifold $X_n$
and one map $\mu$ in order to obtain information on all regular spherical polygon spaces. 
In \cite{K2}, we reproved the results in \cite[\S6]{KM3} along our lines. 

Note that the analogue of the function $\rho_a$ for planar or spatial polygon spaces is well-studied in \cite{KM1} and \cite{KM2}.
On the other hand, the analogue of the function $\mu$ for regular planar or spatial polygon spaces is ineffective. In fact, the topological type of these polygon spaces does not depend on the equilateral side-lengths. 

This paper is organized as follows. In \S2 we state our main results. Theorems A and B give critical points of $\mu$ and their indices, respectively. 
Proposition C, Corollary D and Theorem E
are numerical results for general $a$, which are obtained from Theorems A and B. 
The most typical case is when $a=\pi/2$. We state the results for this case in Theorem F.

In \S3 we first recall the main result in \cite{KM3}, which describes the index of $\rho_a$. (See Theorem \ref{KM3}.)
Then we state Theorem \ref{key}, which  describes the index of $\mu$. 
We deduce the theorem from Theorem \ref{KM3}. 
In \S4 we prove Theorems A and B. In \S5 we prove Proposition C and Corollary D. 
In \S6 we prove Theorem E and in \S7 we prove Theorem F.

\section{Statements of the main results}
\begin{defn}\label{degenerate} A polygon $P \in M_n(a)$ is said to be degenerate if and only if it lies in a great circle $\gamma$ in $S^2$. 
\end{defn}

\begin{thma} An element $(P,a) \in X_n$ is a critical point of $\mu$ if and only if $P$ is degenerate.
\end{thma}

Our next task is to determine the index of a critical point of $X_n$. 
We denote by $\N =\{1,2,3,\cdots\}$ the set of natural numbers.
\begin{thmb}
We set
\begin{equation}\label{gam}
\Gamma_n:= \{(\alpha,\beta) \in \N \times \N \sbar \text{$\alpha$ is odd, $\beta$ is even and $\beta<\alpha \leq n$}\}.
\end{equation}
Then the following assertions hold{\rm :}
\begin{enumerate}
\item To each $(\alpha,\beta) \in \Gamma_n$, there corresponds a certain number of critical points of $\mu$. 
(See \eqref{F} for more precise correspondence.)
All critical points are non-degenerate such that
their information is given by Table \ref{tab1}, where we set 
\begin{equation}\label{set}
\alpha=2s+1 \;\; \text{and} \;\;\beta=2t.
\end{equation}

\begin{table}[H]
\begin{center}
\begin{tabular}{|c|c|c|}\hline
critical value & the number of critical points & the index\\ \hline
$\displaystyle{\frac{\beta}{\alpha} \pi}$ & $\displaystyle{{n \choose m-s}}$ & $\displaystyle{m-s+2t-1}$\\ \hline
\end{tabular}
\end{center}
\caption{The information on $\mu$ at $(\alpha,\beta)$}
\label{tab1}
\end{table}
\item Conversely, a critical point of $\mu$ is attained by a unique $(\alpha,\beta) \in \Gamma_n$. 
\end{enumerate}
\end{thmb}

\begin{rem}\label{parity}
(i) It is not true that critical points of the same critical value have same index. For example, $\Gamma_9$ contains elements
$(3,2)$ and $(9,6)$ such that their critical values are $2\pi/3$. On the other hand, the index of the former is $4$ but that of the latter is $5$. 

(ii) The number of critical points does not depend on $\beta$. Moreover, the index of $\mu$ has the same parity as $m-s-1$, which also does not depend on $\beta$.
\end{rem}

\begin{propc} 
{\rm (i)} Let $U_n$ be the set of critical points of $\mu$. Then we have
\begin{equation}\label{num}
|U_n|=\frac{1}{2} \left(-4^m+\frac{(2m+1)!}{(m!)^2} \right),
\end{equation}
where $| -|$ stands for the cardinality.

{\rm (ii)} Let $V_n$ be the set of critical values of $\mu$. We set
\begin{equation}\label{PHI}
\Phi(n) := \frac{1}{2} \sum_{s=1}^ {m} \varphi(2s+1),
\end{equation}
where $\varphi$ denotes Euler's totient function. 
Then we have $|V_n|= \Phi(n)$. 
\end{propc}

\begin{defn} In what follows, we write
\begin{equation}\label{cv}
V_n=\{\zeta_1,\zeta_2,\cdots,\zeta_{\Phi(n)}\},
\end{equation}
where we arrange $\zeta_i$ in the following order:
$$\zeta_1<\zeta_2<\dots<\zeta_{\Phi(n)}.$$
Note that from the definition of $X_n$, we have $\zeta_1>0$. 
\end{defn}

We study the asymptotic behavior of $|U_n|$ and $|V_n|$ in Proposition C. 
In what follows, the notation
$$f(n) \, \sim \, g(n) \quad (n \to \infty)$$
means that 
$$\lim_{n \to \infty} \frac{f(n)}{g(n)}=1.$$
\begin{cord} We have the following results{\rm :}

{\rm (i)} 
$$|U_n| \,\sim\, 2^{2m-1} \left(-1+\frac{2m+1}{\sqrt{\pi m}}\right) \quad (m \to \infty).$$

{\rm (ii)}
$$|V_n|\, \sim\, \frac{n^2}{\pi^2} \quad (n \to \infty).$$
\end{cord}

Using Theorem B, we can determine $\chi (M_n(a))$ for all $a$. When the subscript $i$ of $\zeta_i$ in \eqref{cv} is near $1$ or $\Phi(n)$, 
we can give $\chi (M_n(a))$ explicitly. In order to state the result, we prepare the following:

\begin{defn}\label{ZE}
We fix an odd number $n$. We define as follows:

(i) Choosing an element $a \in (0,\zeta_1)$, we set $\Omega_0:= \chi (M_n(a))$.

(ii) For $1\leq i \leq \Phi (n)-1$, choosing an element $a \in (\zeta_i,\zeta_{i+1})$, we set $\Omega_i:= \chi (M_n(a))$. 
\end{defn}

\begin{thme} We fix an odd number $n$. 

{\rm (i)} We have 
\begin{equation*}\label{tue1}
\Omega_i = (-1)^{m+1}{2m \choose m}+\frac{(-1)^i2 i}{2m+1} {2m+1 \choose i}
\end{equation*}
for $0 \leq i \leq p$, where we set $p:=\big\lfloor m/2 \big\rfloor +1$.

{\rm (ii)} We have 
\begin{equation*}\label{tue2}
 \Omega_{\Phi(n)-1-i}=\frac{(-1)^i2 (i+1)}{2m+1}  {2m+1 \choose i+1}
\end{equation*}
for $0 \leq i \leq q$, where we set $q:=\big\lfloor 2m/3\big\rfloor$.
\end{thme}

\begin{rem}\label{morse}
The following items are remarks concerning Theorem E.

(i) We have $\Omega_{\Phi(n)-1}=2$, but more is true: 
By the Morse lemma, there is a diffeomorphism
$M_n (a) \cong S^{n-3}$ if $a \in (\zeta_{\Phi(n)-1},\zeta_{\Phi(n)})$. 

(ii)  When the subscript $i$ of $\zeta_i$ in \eqref{cv} is near $1$ or $\Phi(n)$, the value of $\zeta_i$ is given in Lemma \ref{sun}. 
\end{rem}

\begin{ex} 
\begin{align*}
&\Omega_0=  (-1)^{m+1} {2m \choose m}& & \text{for $n\geq 5$}\\
&\Omega_1= -2+ (-1)^{m+1} {2m \choose m}& &\text{for $n \geq 5$}\\
&\Omega_2=4m+ (-1)^{m+1} {2m \choose m}& & \text{for $n \geq 5$}\\
\intertext{and}
&\Omega_{\Phi(n)-1}= 2& & \text{for $n \geq 5$}\\
&\Omega_{\Phi(n)-2}=-2n+2& & \text{for $n \geq 5$}\\
&\Omega_{\Phi(n)-3}= n^2-3n+2& &\text{for $n \geq 7$}.
\end{align*}
\end{ex}

The most typical case is when $a=\pi/2$. To state the results for this case, we prepare a notation: 
Let $\varphi(x,d)$ denote the Legendre totient function defined to be the number of positive integers $\leq x$ which are prime to $d$.
See, for example, \cite[Theorem 261]{HW} for some properties of $\varphi(x,d)$, although we do not need them in this paper. 

\begin{thmf} We consider the element $\zeta_k \in V_n$ such that $\pi/2 \in (\zeta_k, \zeta_{k+1})$. 
\begin{enumerate}
\item The subscript $k$ of $\zeta_k$ is given as follows{\rm :} If we set
\begin{equation}\label{psi}
\psi (n):= \sum_{s=1}^m \varphi (2 \big\lfloor (s+1)/2 \big\rfloor-1,4s+2),
\end{equation}
then we have 
$$k=\Phi(n)-\psi (n),$$
where $\Phi (n)$ is defined in \eqref{PHI}. 
\item For $k$ in {\rm (i)}, the interval $(\zeta_k,\zeta_{k+1})$ is given as follows{\rm :}
\begin{enumerate}
\item When $m$ is odd,
$$(\zeta_k,\zeta_{k+1})= \Bigl( \frac{m-1}{2m-1}\pi, \frac{m+1}{2m+1} \pi \bigr).$$
\item When $m$ is even,
$$(\zeta_k,\zeta_{k+1})= \Bigl( \frac{m}{2m+1}\pi, \frac{m}{2m-1} \pi \bigr).$$
\end{enumerate}
\item For all odd numbers $n$, we have
$$\chi (M_n (\pi/2))= (-1)^{m+1}\cdot 2^{2m-1}.$$
\end{enumerate}
\end{thmf}

\begin{rem}
About Theorem F (iii),  $\chi (M_n (\pi/2))$ for all $n \geq 3$ was determined in \cite{K1} by a different method. 
\end{rem}

\section{Key results}
\begin{thm}\label{surg} Let $h: M \to \R$ be a smooth function on a $d$-dimensional manifold $M$. For numbers $\xi_1<\xi_2$, we assume that $h^{-1}[\xi_1,\xi_2]$ 
is compact and contains a unique non-degenerate critical point $p$ of index $r$. Then the following results hold{\rm :}
\begin{enumerate}
\item The level set $h^{-1}(\xi_1)$ is obtained from $h^{-1}(\xi_2)$ by removing $\text{\rm Int} \,(\partial D^{d-r}\times D^r)$ and 
attaching $ D^{d-r} \times \partial D^r$ along the boundary. We call this construction a surgery of type $(d-r, r)$
\item We have $$\chi (h^{-1}(\xi_1))=\chi (h^{-1}(\xi_2))+2 (-1)^{r+1}.$$
\item We set $\xi_3:=h(p)$. Then level set $h^{-1}(\xi_3)$ is obtained from $h^{-1}(\xi_2)$ by removing 
$\text{\rm Int} \,(\partial D^{d-r}\times D^r)$ and attaching $C(S^{d-r-1}\times S^{r-1})$ along the boundary, 
where $C$ denotes the cone. In particular, we have 
$$\chi (h^{-1}(\xi_3))=\chi (h^{-1}(\xi_2))+ (-1)^{r+1}.$$
\end{enumerate}
\end{thm}

\begin{proof}
The theorem is well-known in Morse surgery. (See, for example, \cite{M3}.)
\end{proof} 

Next we recall results in \cite{KM3}. Let $(u,v)$ denote the shortest geodesic segment connecting non-antipodal points $u, v$ in $S^2$.
For $P=(u_1,\dots,u_n) \in M_n (a)$, we set $e_i:= (u_i,u_{i+1})$ for $1 \leq i \leq n-1$ and $e_n:= (u_n,u_1)$. 

Suppose now that $P \in M_n(a)$ is a degenerate polygon contained in a great circle $\gamma$. Orient $\gamma$ so that $e_n$ is 
negatively directed. 
We say that $e_i$ is {\it forward-track} if the orientation of $e_i$ agrees with that of $\gamma$. Otherwise, we say that $e_i$ is {\it back-track}.
We let $f=f(P)$ be the number of forward-tracks and $b=b(P)$ be the number of back-tracks so that 
\begin{equation}\label{10}
f+b=n.
\end{equation}
Define the winding number $w=w(P)$ by
\begin{equation}\label{11}
a (f-b)=2\pi w.
\end{equation}

\begin{rem}\label{en}
Because of the way of giving orientation to $\gamma$, $e_n$ is always back-track. 
\end{rem}

We state the main result in \cite{KM3} in the form which we need in this paper. Recall that the function $\rho_a: N_n(a) \to \R$ is defined in \eqref{g}. 
\begin{thm}[{\cite[Main Theorem]{KM3}}] \label{KM3} Let $P \in M_n(a)$ be a degenerate polygon. Forgetting the condition $d(u_n,u_1)=a$, 
$P$ naturally defines an element $T \in N_n(a)$. 
Then the signature of $D^2 \rho_a |_T$ is given by
$$(b(P)+2 w(P)-1,f(P)-2w(P)-1).$$
\end{thm}

Recall that the map $\mu:X_n \to \R$ is defined in \eqref{4}. 
The following theorem is a key to proving Theorem B and indicates that the index of $\mu$ is more subtle than that of $\rho_a$: 
\begin{thm}\label{key} Let $P \in M_n (a)$ be a degenerate polygon. Then the following results hold{\rm :}
\begin{enumerate}
\item When $w(P)>0$, the signature of $D^2 \mu|_{(P,a)}$ is given by
$$(f(P)-2w(P)-1,b(P)+2w(P)-1).$$
\item When $w(P)<0$, the signature of $D^2 \mu|_{(P,a)}$ is given by
$$(b(P)+2w(P)-1,f(P)-2w(P)-1).$$
\end{enumerate}
\end{thm}

In order to prove Theorem \ref{key}, we first prove the following:
\begin{lem}\label{diff} Let $O$ be an open neighborhood of $(P,a)$ in $X_n$ such that $O$ contains no other degenerate polygons than $(P,a)$. 
We fix a sufficiently small positive real number $\varepsilon$. Then the following assertions hold{\rm :}

{\rm (i)} Assume that $w(P)>0$. Then for all $j \in \{-1,1\}$, 
$\mu^{-1}(a +j \varepsilon)\cap O$ is diffeomorphic to an open set of $\rho_a^{-1} (a-j\varepsilon).$

{\rm (ii)} Assume that $w(P)<0$. Then for all $j \in \{-1,1\}$, 
$\mu^{-1}(a +j \varepsilon)\cap O$ is diffeomorphic to an open set of $\rho_a^{-1} (a+j\varepsilon).$
\end{lem}

\begin{proof}[Proof of Lemma \ref{diff}]
(i) We write the side-lengths of a spherical polygon as $(r_1, r_2,\dots,r_n)$. For $\delta \in [0,\varepsilon]$, we deform the side-lengths of an element
of $\mu^{-1} (a+\varepsilon) \cap O$ by 
\begin{equation}\label{edge}
(a+j\delta,a+j\delta,\cdots,a+j\delta, a+j(-\varepsilon+2\delta))
\end{equation}
We need to check that this deformation is indeed possible. To see this, it will suffice to see that \eqref{edge} does not cross a wall for any
$\delta$, where a wall is defined in \cite[p.311]{KM3}. 

Recall that $e_n$ is always back-track. (See Remark \ref{en}.) 
Hence we have about \eqref{edge} that 
\begin{align}\label{u}
&\text{the sum of forward-track side-lengths}-\text{the sum of  back-track side-lengths} \notag\\
&\phantom{\text{the sum of forward}} =2\pi w(P)+j \left( \varepsilon+(u-2)\delta\right)
\end{align}
for some integer $u$. We claim that $u\geq 2$. In fact, the assumption $w(P)>0$ tells us that at least 
$m+1$ of the first $2m$ components of \eqref{edge}
are forward-track. Hence we have $u\geq (m+1)-(m-1)=2$ and the claim follows.

Now the fact that $\varepsilon+ (u-2)\delta>0$ in \eqref{u} implies that \eqref{edge} does not cross a wall for any $\delta$. 
Finally, setting $\delta=0$ or $\varepsilon$ in \eqref{edge}, we complete the proof of (i). 

(ii) Instead of \eqref{edge}, we deform the side-lengths of an element
of $\mu^{-1} (a+\varepsilon) \cap O$ by 
\begin{equation}\label{edge2}
(a+j\delta,a+j\delta,\cdots,a+j\delta, a+j\varepsilon).
\end{equation}
we have about \eqref{edge2} that 
\begin{align}\label{uu}
&\text{the sum of forward-track side-lengths}-\text{the sum of  back-track side-lengths} \notag\\
&\phantom{\text{the sum of forward}} =2\pi w(P)-j \left(\varepsilon+v\delta\right)
\end{align}
for some integer $v$. We claim that $v \geq 0$. In fact,  the assumption $w(P)<0$ tells us that at most
$m$ of the first $2m$ components of \eqref{edge2}
are forward-track. Hence we have $-v\leq m-m=0$ and the claim follows.

Now the fact that $\varepsilon+ v\delta>0$ in \eqref{uu} implies that \eqref{edge2} does not cross a wall for any $\delta$. 
Finally, setting $\delta=0$ or $\varepsilon$ in \eqref{edge}, we complete the proof of (ii). 
This completes the proof of Lemma \ref{diff}. 
\end{proof}

\begin{proof}[Proof of Theorem \ref{key}] 
We apply Theorem \ref{surg} (i) to the map $\mu$. Note that $\dim X_n=n-2$. We denote by $r$ the index of $P$. 
When the level set $\mu^{-1} (a+\varepsilon)$ descends to 
$\mu^{-1}(a-\varepsilon)$, the critical point $P$ gives a surgery of type 
\begin{equation}\label{type}
(n-2-r,r).
\end{equation}
We set $\lambda:=f(P)-2w(P)-1$. 

(i) If $w(P)>0$, then Lemma \ref{diff} (i) tells us that the above descent is equivalent to the ascent from $\rho_a^{-1}(a-\varepsilon)$ to $\rho_a^{-1}(a+\varepsilon)$.
Combining Theorem \ref{surg} (i) and Theorem \ref{KM3}, when we cross through $P$, a surgery of type 
\begin{equation}\label{equi}
(\lambda,n-2-\lambda)
\end{equation}
occurs. 
Comparing \eqref{type} and \eqref{equi}, we ave
$$r=n-2-\lambda= b(P)+2w(P)-1.$$

(ii) If $w(P)<0$, then the descent from $\mu^{-1} (a+\varepsilon)$ to $\mu^{-1}(a-\varepsilon)$ is equivalent to the descent from
 from $\rho_a^{-1}(a+\varepsilon)$ to $\rho_a^{-1}(a-\varepsilon)$. When we cross through $P$, a surgery of type 
\begin{equation}\label{equi1}
(n-2-\lambda,\lambda)
\end{equation}
occurs. Comparing \eqref{type} and \eqref{equi1}, we have
$$r=\lambda= f(P)-2w(P)-1.$$
This completes the proof of Theorem \ref{key}. 
\end{proof}

\section{Proofs of Theorems A and B}
\begin{proof}[Proof of Theorem A] 
We prove the theorem along the lines of \cite[Theorem 2.9]{KM3}. Combining the following three assertions, we obtain Theorem A: 
\begin{enumerate}
\item (An analogue of \cite[Lemma 2.7 (ii)]{KM3}.) By \eqref{5}, the Zariski tangent space $T_P M_n(a)$ is given by
$$T_P M_n(a)= \text{ker}\, d\mu |_{(P,a)}.$$
\item (An analogue of \cite[Corollary 2.8]{KM3}.) We see from (i) that a point $(P,a)$ is a singular point of $X_n$ if and only of $(P,a)$ is a critical point of $\mu$.
\item By \cite[Theorem 1.1]{KMRIMS}, $P$ is a singular point of $M_n(a)$ if and only if $P$ is degenerate. 
\end{enumerate}
\end{proof}

\begin{proof}[Proof of Theorem B]
(i) We define the map
\begin{equation}\label{F}
F: U_n \to \Gamma_n
\end{equation}
as follows, where $U_n$ is defined in Proposition C (i): For $(P,a) \in X_n$, we set
\begin{equation}\label{cas}
F(P,a):=
\begin{cases}
(f(P)-b(P),2w(P)) &\quad \text{if $w(P)>0$}\\
(-f(P)+b(P), -2w(P)) & \quad \text{if $w(P)<0$.}
\end{cases}
\end{equation}
Note that \eqref{10} and \eqref{11} tell us that $F$ is certainly a map to $\Gamma_n$. Hence, in the notation of Theorem B (i), 
we can set $(\alpha,\beta):= F(P,a)$. Now we check the items (i) and (ii) of Theorem B below.

In what follows, we prove Table \ref{tab1}. First, \eqref{11} implies that $a=\beta\pi /\alpha$ for $w$ positive or negative.
Hence the critical value in Table \ref{tab1} is true. 

Second, we compute the index of $\mu$ at $(\alpha,\beta)$. Using \eqref{set} and \eqref{cas}, we have
\begin{align}
(f,b)=&\begin{cases}
(m+s+1,m-s), &\quad \text{if $w>0$},\\
(m-s,m+s+1), & \quad \text{if $w<0$}
\end{cases} \label{18}\\
\intertext{and}
w=&\begin{cases}
t, &\quad \text{if $w>0$},\\
-t,&\quad \text{if $w<0$}.
\end{cases}\label{19}
\end{align}
Using Theorem \ref{key}, \eqref{18} and \eqref{19}, we have 
$$\text{the index of $\mu$ at $(\alpha,\beta)$ is $ m-s+2t-1$}$$
for $w$ positive or negative. 
Hence the index of $\mu$ in Table \ref{tab1} is true. 

Third, we compute the number of critical points. 
\eqref{cas} tells us that critical points $(P, a)$ and $(Q,a)$ satisfy $F(P,a)=F(Q,a)$
if and only if $(f(P),b(P))=(f(Q),b(Q))$ or $(f(P),b(P))=(b(Q),f(Q))$. The description of $f$ in terms of $s$ is computed in \eqref{18}. Since $e_n$ is always
back-track, we need to choose $f$-elements from $\{e_1,e_2,\dots,e_{2m}\}$. The total number of such choices is
\begin{align*}
{2m \choose m+s+1}+{2m \choose m-s}=&{2m \choose m-s-1}+{2m \choose m-s}\\
=&{2m+1 \choose m-s}.
\end{align*}
Hence the number of critical points in Table \ref{tab1} is true. This completes the proof of Theorem B (i).

(ii) The item is already proved in the third part of the above proof of Theorem B (i). 
\end{proof}

\section{Proofs of Proposition C and Corollary D}
\begin{proof}[Proof of Proposition C]
(i) We consider Table \ref{tab1}. If we fix $\alpha=2s+1$, then the number of choices of $\beta$ is $s$. Moreover, if we fix $(\alpha,\beta)$, then
the number of the corresponding critical points is ${n \choose m-s}$. Varying $s$ in $1 \leq s \leq m$, we have
$$|U_n|=\sum_{s=1}^m s {n \choose m-s}.$$
Using \cite[A000531]{OEIS}), we obtain \eqref{num}. 

(ii) We set
$$\widetilde{\Gamma}_n := \{(\alpha,\beta) \in \Gamma_n \sbar \text{$\alpha$ and $\beta$ are coprime}\}.$$
Table \ref{tab1} tells us that 
\begin{equation}\label{c1}
|V_n|=|\widetilde{\Gamma}_n|.
\end{equation}
Let $p_1:\widetilde{\Gamma}_n\to \N$ be the projection to the first factor. We claim that for each $\alpha=2s+1$, we have
$|p_1^{-1}(\alpha)|=\varphi (\alpha)/2$. In fact, if we forget the condition ``$\beta$ is even" for $p_1^{-1}(\alpha)$, then its cardinality is
$\varphi (\alpha)$. Moreover, since $\alpha$ is odd, exactly one of $\beta$ and $\alpha-\beta$ is even for all $\beta \in \N$.
Namely, exactly one of $(\alpha,\beta)$ and $(\alpha,\alpha-\beta)$ is an element of $p_1^{-1}(\alpha)$. 
Hence the claim follows. 

Now varying $s$ in $1 \leq s \leq m$, we have 
\begin{equation}\label{c2}
|\widetilde{\Gamma}_n|=\sum_{s=1}^m \varphi(2s+1)/2.
\end{equation}
Combining \eqref{c1} and \eqref{c2}, we obtain Theorem C (ii).
\end{proof}

\begin{proof}[Proof of Corollary D]
(i) The item is a consequence of the following Stirling's formula:
$$k! \, \sim \, \sqrt{2\pi k} \left( \frac{k}{e} \right)^k \quad (k \to \infty).$$

(ii) It will suffice to prove the following:
\begin{lem}\label{asym}
We set 
$$\Psi(n):= \sum_{\substack{i=1 \\ \text{$i$ odd}}}^n \varphi(i).$$
Then we have 
\begin{equation}\label{har}
\Psi(n) \, \sim \, \frac{2n^2}{\pi^2} \quad (n \to \infty).
\end{equation}
\end{lem}

\begin{proof}[Proof of Lemma \ref{asym}]
To the best of the author's knowledge, there is no publication which proves the lemma. Hence we give a proof here.
We modify the proof of \cite[Theorem 330]{HW}, which claims that the average order of $\varphi(n)$ is $6n/\pi^2$. 
\begin{align}
\Psi(n) &= \sum_{\substack{i=1 \\ \text{$i$ odd}}}^n i  \sum_{d|i} \frac{\mu(d)}{d}
=\sum_{\substack{dd' \leq n \\ \text{$dd'$ odd}}} d' \mu(d) \notag\\
&= \sum_{\substack{d=1 \\ \text{$d$ odd}}}^n \mu(d) \sum_{\substack{d'=1 \\ \text{$d'$ odd}}}^{\big\lfloor n/d \big\rfloor} d'
= \sum_{\substack{d=1 \\ \text{$d$ odd}}}^n \mu(d) \left( \Big\lfloor\frac{\big\lfloor n/d \big\rfloor-1}{2} \Big\rfloor+1 \right)^2 \notag\\
&\sim \frac{n^2}{4} \sum_{\substack{d=1 \\ \text{$d$ odd}}}^n \frac{\mu(d)}{d^2} \sim 
  \frac{n^2}{4} \sum_{\substack{d=1 \\ \text{$d$ odd}}}^{\infty} \frac{\mu(d)}{d^2}= \frac{n^2}{4 \widetilde{\zeta}(2)} \label{23},
\end{align}
where we define 
$$\widetilde{\zeta}(2):= \sum_{\substack{j=1 \\ \text{$j$ odd}}}^{\infty} \frac{1}{j^2}.$$
Since
$$\zeta(2)= \widetilde{\zeta}(2)+ \frac{1}{4} \zeta(2),$$
we have 
\begin{equation}\label{wide}
\widetilde{\zeta}(2) = \frac{\pi^2}{8}.
\end{equation}
Substituting \eqref{wide} into \eqref{23}, we obtain \eqref{har}. This completes the proof of Lemma \ref{asym}, 
and hence also that of Corollary D (ii).
\end{proof}
\renewcommand{\qedsymbol}{}
\end{proof}

\section{Proof of Theorem E}
Throughout this section, let $p$ and $q$ be as in Theorem E (i) and (ii), respectively. 
We first prove the following:
\begin{lem}\label{sun} We fix an odd number $n$.

{\rm (i)} For $1 \leq i \leq p$, we have
\begin{equation}\label{ze1}
\zeta_i= \frac{2}{2m-2i+3}\pi.
\end{equation}
Moreover, if an element $(\alpha,\beta)$ of $ \Gamma_n$ satisfies that $\beta\pi/\alpha=\eqref{ze1}$, then $\alpha$ and $\beta$ are coprime.
Thus
\begin{equation}\label{30}
(\alpha,\beta)=(2m-2i+3,2).
\end{equation}

{\rm (ii)} For $\Phi(n)-q \leq i \leq \Phi(n)$, we have
\begin{equation}\label{ze2}
\zeta_i= \frac{2m-2\Phi(n)+2i}{2m-2\Phi(n)+2i+1}\pi.
\end{equation}
Moreover,  if an element $(\alpha,\beta)$ of $ \Gamma_n$ satisfies that $\beta\pi/\alpha=\eqref{ze2}$, then $\alpha$ and $\beta$ are coprime.
Thus
\begin{equation}\label{31}
(\alpha,\beta)=(2m-2\Phi(n)+2i+1,2m-2\Phi(n)+2i).
\end{equation}

\end{lem}
\begin{proof} 
Since the proof of the lemma is easy, we check only the reason why the bounds $i\leq p$ and $\Phi(n)-q \leq i$ appear. 
It is easy to see that if we write $V_n$ in \eqref{cv} in the ascending order, it has the following form:
\begin{align*}
&V_n= \biggl\{ \text{several elements of the form \eqref{ze1}},\;\frac{4}{2m+1}\pi, \cdots,\biggr.\\
&\biggl.\phantom{several elements of} \frac{2m-2}{2m+1}\pi, \;\text{several elements
of the form \eqref{ze2}}\biggr\}.
\end{align*}

(i) In order that the inequality 
$$\eqref{ze1} <\frac{4}{2m+1}\pi$$
holds, we must have $i \leq p$. 

(ii) In order that the inequality
$$\frac{2m-2}{2m+1}\pi< \eqref{ze2}$$
holds, we must have $\Phi(n)-(2m+1)/3< i$. Namely, we must have $\Phi(n)-q\leq i$. 
\end{proof}

\begin{prop}\label{rec} We fix an odd number $n$. 

{\rm (i)} We define the recurrence relation $\{\sigma_i\}_{i=0}^\infty$ as follows{\rm :}
\begin{align*}\label{s1}
&\sigma_0= (-1)^{m+1} {2m \choose m} \quad \text{and}\\
&\sigma_{i+1}= \sigma_{i}+2 (-1)^{i +1} {n \choose i} \quad \text{for $i\geq 0$}.
\end{align*}
Then we have 
\begin{equation}\label{om1}
\sigma_i= \Omega_i
\;\; \text{
for \;$0 \leq i \leq p$.}
\end{equation}

{\rm (ii)} We define the recurrence relation $\{\tau_i\}_{i=0}^\infty$ as follows{\rm :}
\begin{align*}\label{s1}
&\tau_0=2 \quad \text{and}\\
&\tau_{i+1}= \tau_{i}+2 (-1)^{i+1}  {n \choose i+1} \quad \text{for $i\geq 0$}.
\end{align*}
Then we have 
\begin{equation}\label{om2}
\tau_i= \Omega_{\Phi(n)-1-i}
\;\; \text{
for \;$0 \leq i \leq q$.}
\end{equation}
\end{prop}

\begin{proof}
(i) We prove \eqref{om1} by induction on $i$. 

(a) For the case $i=0$, using Table \ref{tab1} and Remark \ref{parity} (ii), we have 
\begin{align*}
\chi (M_n (a))=&2 \sum_{s=1}^m (-1)^{m+s} s {n \choose m-s}\\
=&(-1)^{m+1} {2m \choose m}.
\end{align*}

(b) Assume that \eqref{om1} holds for $i=k$. In order to complete the induction, we need to prove 
\begin{equation}\label{ind1}
\Omega_{k+1}=\Omega_k+2  (-1)^{k+1}  {n \choose k}.
\end{equation}
We consider Theorem \ref{surg}  for the situation that the level set of the function \eqref{4} crosses through $\mu^{-1} (\zeta_{k+1})$. 
If we set $i=k+1$ in \eqref{30}, we obtain the element $(\alpha,\beta)$ which satisfies that $\beta\pi/\alpha= \zeta_{k+1}$. 
In the notation of Theorem B (i), this $(\alpha,\beta)$ is given by 
\begin{equation*}\label{ST}
s= m-k \quad \text{and} \quad t=1.
\end{equation*}
Substituting this into Table \ref{tab1}, we see that the number of critical points is ${n \choose k}$ and the index is $k+1$. 
Using Theorem \ref{surg} (ii), we have
$$\Omega_k = \Omega_{k+1}+ 2  (-1)^{k+2}  {n \choose k}.$$
Thus we obtain \eqref{ind1}.

(ii) We prove \eqref{om2} by induction on $i$.

(a) The case for $i=0$ follows from Table \ref{tab1} easily. We may also deduce the case from the Morse lemma. 
(See Remark \ref{morse} (i).)

(b) Assume that \eqref{om2} holds for $i=k$. In order to complete the induction, we need to prove 
\begin{equation}\label{ind2}
\Omega_{\Phi(n)-k-2}=\Omega_{\Phi(n)-k-1}+2  (-1)^{k+1}  {n \choose k+1}.
\end{equation}
We consider the situation that the level set of the function \eqref{4} crosses through 
$\mu^{-1} (\zeta_{\Phi(n)-k-1})$. 
If we set $i=\Phi(n)-k-1$ in \eqref{31}, we obtain the element $(\alpha,\beta)$ which satisfies that $\beta\pi/\alpha= \zeta_{\Phi(n)-k-1}$. 
In the notation of Theorem B (i), this $(\alpha,\beta)$ is given by 
\begin{equation*}\label{ST1}
s=t=m-k-1.
\end{equation*}
Substituting this into Table \ref{tab1} , we see that the number of critical points is ${n \choose k+1}$ and the index is $2m-k-2$. 
Using Theorem \ref{surg} (ii), we obtain \eqref{ind2}.
\end{proof}

\begin{proof}[Proof of Theorem E]
Theorem E follows easily from Proposition \ref{rec}. 
\end{proof}

\section{Proof of Theorem F}
Recall that $V_n$ is defined in \eqref{cv}. 
Let us write up elements $\zeta_i \in V_n$ such that $\pi/2<\zeta_i$. 
We set 
\begin{equation}\label{theta}
\theta (n):=\left\{\frac{2t}{2s+1}\pi \longsbar 1 \leq s \leq m, \;\  \big\lfloor s/2\big\rfloor +1 \leq t \leq s\right\}. 
\end{equation}
It is clear that $\zeta_i \in V_n$ satisfies $\pi/2 < \zeta_i$ if and only if $\zeta_i \in \theta(n)$. 
In order to prove Theorem F (i), we need to determine $|\theta(n)|$, the number of irreducible fractions in $\theta(n)$. 
For that purpose, we first prove the following: 
\begin{lem}\label{lem1}
For a fixed $s$, we put
$$J_s:=\left\{t \in \N \sbar \big\lfloor s/2 \big\rfloor +1 \leq t\leq s \; \text{and $t$ is prime to $2s+1$} \right\}.$$
Then we have 
$$|J_s |= \varphi (2 \big\lfloor (s+1)/2 \big\rfloor-1,4s+2),$$ 
where $\varphi (x,d)$ is defined above Theorem F. 
\end{lem}
\begin{proof}[Proof of Lemma \ref{lem1}]
First, we set
\begin{equation*}
K_s:=\left\{i \in \N\sbar  1 \leq i \leq  2\big\lfloor (s+1)/2 \big\rfloor-1, \; \text{$i$ is odd and prime to $2s+1$}\right\}.
\end{equation*}
We define the map $G: J_s \to K_s$ by $G(t):=2s+1-2t$. Then it is easy to see that $G$ is a bijection. 

Second, we set
$$L_s:=\left\{ i\in \N\sbar  1 \leq i \leq 2\big\lfloor (s+1)/2 \big\rfloor-1 \; \text{and $i$ is prime to $4s+2$}\right\}.$$
Since an element of $K_s$ is odd, we have 
$K_s = L_s$. 
Moreover, from the definition of $\varphi(x,d)$, we have $|L_s|= \varphi (2 \big\lfloor (s+1)/2 \big\rfloor-1,4s+2).$
Combining the above results, Lemma \ref{lem1} follows. 
\end{proof}

\begin{proof}[Proof of Theorem F]
(i) Lemma \ref{lem1} tells us that $|\theta(n)|= \psi(n)$, where $\psi (n)$ is defined in \eqref{psi}. 
Since the number of critical values of $\mu$ is $\Phi(n)$ by Proposition C (ii), we obtain Theorem F (i). 

(ii) For $\theta(n)$ in \eqref{theta}, we set $\theta(n)^c := V_n \setminus \theta(n)$. 
Then $\zeta_k$ and $\zeta_{k+1}$ for $k$ as given in Theorem F (i) are specified as follows: 
\begin{equation}\label{max}
\zeta_k = \text{max}\; \theta(n)^c \quad \text{and} \quad \zeta_{k+1} = \text{min}\; \theta(n).
\end{equation}
Computing \eqref{max} explicitly, we obtain (ii). 

(iii) Using Theorem B, Remark \ref{parity} (ii), Theorem \ref{surg} (ii) and  \eqref{theta}, we have
\begin{align}
\chi (M_n (\pi/2))=&2 \sum_{s=1}^m (-1)^{m+s} \left(s- \big\lfloor s/2 \big\rfloor \right)  {n \choose m-s}\notag \\
=& 2(-1)^m \sum_{s=1}^m (-1)^s \big\lfloor (s+1)/2 \big\rfloor {2m+1 \choose m-s}\notag \\
=& 2 (-1)^{m+1} \sum_{i=0}^{m+1} (-1)^i\big\lfloor i /2\big\rfloor {2m+1 \choose m+i} \label{29}.
\end{align}
It is proved in \cite{K11} that 
$$ \sum_{i=0}^{m+1} (-1)^i\big\lfloor i /2\big\rfloor {2m+1 \choose m+i}=2^{2m-2}.$$
Hence we have from \eqref{29} that $\chi (M_n(\pi/2))=(-1)^{m+1}\cdot  2^{2m-1}.$
\end{proof}

\end{document}